\documentclass{amsart}
\usepackage{amssymb,amsfonts,latexsym,amscd}

\usepackage[all]{xy}
\usepackage{verbatim}

\newtheorem{theorem}{Theorem}[section]
\newtheorem{lemma}[theorem]{Lemma}

\theoremstyle{definition}
\newtheorem{definition}[theorem]{Definition}

\newtheorem{remark}[theorem]{Remark}

\newcommand{\id}{\text{id}}

\newcommand{\Aut}{\text{Aut}}

\newcommand{\Rep}{\text{Rep}}
\newcommand{\Vect}{\text{Vec}}

\newcommand{\g}{\mathfrak{g}}

\renewcommand{\b}{\mathfrak{b}}

\newcommand{\ot}{\otimes}

\newcommand{\ben}{\begin{enumerate}}
\newcommand{\een}{\end{enumerate}}

\newcommand{\C}{{\mathcal C}}
\newcommand{\D}{{\mathcal D}}

\newcommand{\E}{\mathcal{E}}

\newcommand{\Z}{{\mathcal Z}}

\begin{document}

\title[The small quantum group as a quantum double]
{The small quantum group as a quantum double}
\author{Pavel Etingof}
\address{Department of Mathematics, Massachusetts Institute of Technology,
Cambridge, MA 02139, USA} \email{etingof@math.mit.edu}

\author{Shlomo Gelaki}
\address{Department of Mathematics, Technion-Israel Institute of
Technology, Haifa 32000, Israel} \email{gelaki@math.technion.ac.il}

\date{\today}




\begin{abstract}
We prove that the quantum double of the quasi-Hopf algebra
$A_q(\g)$ of dimension $n^{\dim \g}$ attached in \cite{eg} to
a simple complex Lie algebra $\g$ and a primitive root of unity
$q$ of order $n^2$ is equivalent to Lusztig's small quantum group
${\mathfrak u}_q(\g)$ (under some conditions on $n$).
We also give a conceptual construction of $A_q(\g)$ using the
notion of de-equivariantization
of tensor categories.
\end{abstract}

\maketitle

\section{Introduction}

It is well known from the work of Drinfeld \cite{D} that the quantum
group $U_q(\g)$ attached to a simple complex Lie algebra $\g$ can be
produced by the quantum double construction. Namely, the quantum
double of the quantized Borel subalgebra $U_q(\b)$ is the product of
$U_q(\g)$ with an extra copy of the Cartan subgroup $U_q({\mathfrak
h})$, which one can quotient out and get the pure $U_q(\g)$. This
principle applies not only to quantum groups with generic $q$, but
also to Lusztig's small quantum groups at roots of unity,
${\mathfrak u}_q(\g)$ (\cite{L1,L2}). However, ${\mathfrak u}_q(\g)$
itself (without an additional Cartan) is not, in general, a quantum
double of anything: indeed, its dimension is $d=m^{\dim \g}$ (where
$m$ is the order of $q$), which is not always a square.

However, in the case when $m=n^2$ (so that the dimension $d$ is a
square), we have introduced in \cite{eg}, Section 4, a quasi-Hopf
algebra $A_q=A_q(\g)$ of dimension $d^{1/2}$, constructed out of a
Borel subalgebra ${\mathfrak b}$ of $\g$. So one might suspect that
the quantum double of $A_q(\g)$ is twist equivalent to ${\mathfrak
u}_q(\g)$. This indeed turns out to be the case (under some
conditions on $n$), and is the main result of this note. In other
words, our main result is that the Drinfeld center ${\mathcal
Z}(\Rep(A_q(\g)))$ of the category of representations of $A_q(\g)$
is $\Rep ({\mathfrak u}_q(\g))$.

We prove our main result by showing that
the category $\Rep({\mathfrak u}_q({\mathfrak b}))$ of representations
of the quantum Borel subalgebra ${\mathfrak{u}}_q({\mathfrak b})$ is the equivariantization
of the category $\Rep(A_q(\g))$ with respect to an action of a
certain finite abelian group. Thus, $\Rep(A_q(\g))$ can be
conceptually defined as a
de-equivariantization of $\Rep({\mathfrak u}_q(\g))$. So,
one may say that the main outcome of this paper is
a demystification of the quasi-Hopf algebra $A_q(\g)$
constructed ``by hand'' in \cite{eg}.

The structure of the paper is as follows. In Section 2 we recall the
theory of equivariantization and de-equivariantization of tensor
categories. In Section 3 we recall the construction of the
quasi-Hopf algebra $A_q(\g)$ from the paper \cite{eg}. In Section 4
we state the main results. Finally, Section 5 contains proofs.

{\bf Acknowledgments.} The research of the first author was
partially supported by the NSF grant DMS-0504847.
The second author was supported by The Israel
Science Foundation (grant No. 125/05).
Both authors were supported by BSF grant No. 2002040.

\section{Equivariantization and de-equivariantization}

The theory of equivariantization and de-equivariantization of tensor
categories was developed in \cite{B,M} in the setting of fusion
categories; it is now a standard technique in the theory of fusion
categories, and has also been used in the setting of the Langlands
program \cite{F}. A detailed description of this theory is given in
\cite{DGNO} (see also \cite{ENO}, Sections 2.6 and 2.11). This
theory extends without major changes to the case of finite tensor
categories (as defined in \cite{eo}), i.e, even if the
semisimplicity assumption is dropped. Let us review the main
definitions and results of this theory.

\subsection{Group actions}
Let $\C$ be a finite tensor category (all categories and algebras
in this paper are over $\Bbb C$). Consider the category
$\underline{\Aut}(\C)$, whose objects are tensor auto-equivalences
of $\C$ and whose morphisms are isomorphisms of tensor functors. The
category $\underline{\Aut}(\C)$ has an obvious structure of a
monoidal category, in which the tensor product is the composition of
tensor functors.

Let $G$ be a group, and let $\underline{G}$ denote the category
whose objects are elements of $G$, the only morphisms are the
identities and the tensor product is given by multiplication in $G$.

\begin{definition} \label{action}
An {\em action} of a group $G$ on a finite tensor category $\C$ is a
monoidal functor $\underline{G}\to \underline{\Aut}(\C)$.

If $\C$ is equipped with a braided structure we say
that an action $\underline{G}\to \underline{\Aut}(\C)$
respects the braided structure if the image of $\underline{G}$ lies
in $\underline{\Aut}^{br}(\C)$, where
$\underline{\Aut}^{br}(\C)$ is the full subcategory of
$\underline{\Aut}(\C)$ consisting of braided equivalences.
\end{definition}

\subsection{Equivariantization.}
Let a finite group $G$ act on a finite tensor
category $\C$.
For any $g\in G$ let $F_g\in
\underline{\Aut}(\C)$ be the corresponding functor and for any $g,h\in G$ let $\gamma_{g,h}$
be the isomorphism $F_g\circ F_h\simeq F_{gh}$ that defines the
tensor structure on the functor $\underline{G}\to \underline{\Aut}(\C)$.
A {\em $G$-equivariant object\,} of $\C$ is
an object $X\in \C$ together with isomorphisms $u_g: F_g(X)\simeq X$ such that the diagram
\begin{equation*}
\label{equivariantX}
\xymatrix{F_g(F_h(X))\ar[rr]^{F_g(u_h)} \ar[d]_{\gamma_{g,h}(X)
}&&F_g(X)\ar[d]^{u_g}\\ F_{gh}(X)\ar[rr]^{u_{gh}}&&X}
\end{equation*}
commutes for all $g,h\in G$.
One defines morphisms of equivariant objects to be morphisms in $\C$ commuting with $u_g,\;
g\in G$. The category of $G$-equivariant objects of $\C$ will be denoted by $\C^G$. It is
called the {\bf equivariantization} of $\C$.

Note that $\Vect^G=\Rep(G)$, so there is a natural inclusion $\iota:
\Rep (G)\to \C^G$.

One of the main results about equivariantization is the following
theorem (see \cite{ENO}, Proposition 2.10 for the semisimple
case; in the non-semisimple situation, the proof is parallel).

\begin{theorem}\label{help1}
Let $G$ be a finite group acting on a finite tensor category $\C$. Then
$\Rep(G)$ is a Tannakian subcategory of the Drinfeld center $\Z
(\C^G)$ (i.e., the braiding of $\Z(\C^G)$
restricts to the usual symmetric braiding of $\Rep(G)$), and the composition
$$
\Rep(G)\to \Z(\C^G)\to \C^G
$$
(where the last arrow is the forgetful functor) is the natural
inclusion $\iota$.

If $\C$ is a braided category, and
the $G$-action preserves the braided structure,
then $\C^G$ is also braided.
Thus $\C^G$ is a full subcategory of $\Z(\C^G)$,
and the inclusion $\iota$ factors through $\C^G$.
Thus in this case $\Rep(G)$ is a Tannakian subcategory of
$\C^G$.
\end{theorem}

\subsection{De-equivariantization}
Let $\D$ be a finite tensor category such that the Drinfeld center
$\Z(\D)$ contains a Tannakian subcategory $\Rep(G)$, and the
composition $\Rep(G)\to \Z(\D)\to \D$ is an inclusion. Let $A:={\rm
Fun}(G)$ be the algebra of functions $G\to \Bbb C$. The group $G$
acts on $A$ by left translations, so $A$ can be considered as an
algebra in the tensor category $\Rep(G)$, and thus as an algebra in
the braided tensor category $\Z(\D)$. As such, the algebra $A$ is
braided commutative. Therefore, the category of $A$-modules in $\D$
is a tensor category, which is called the {\bf
de-equivariantization} of $\D$ and denoted by $\D_G$.

Let us now separately consider de-equivariantization of braided categories.
Namely, let $\D$ be a finite braided tensor category, and $\Rep(G)\subset
\D$ a Tannakian subcategory. In this case $\Rep(G)$ is also a Tannakian subcategory
of the Drinfeld center $\Z(\D)$ (as $\D\subset \Z(\D)$), so we can define the de-equivariantization
$\D_G$. It is easy to see that $\D_G$ inherits the braided structure from $\D$,
so it is a braided tensor category.

We will need the following result
(see \cite{ENO}, Section 2.6 and Proposition 2.10 for the semisimple
case; in the non-semisimple situation, the proof is parallel).

\begin{theorem}\label{help2}
(i) The procedures of equivariantization and de-equivariantization
are inverse to each other.

(ii) Let $\C$ be a finite tensor category with an action of a finite group $G$.
Let $\E'$ be the M\"uger centralizer of $\E=\Rep(G)$ in $\Z(\C^G)$ (i.e., the category of objects
$X\in \Z(\C^G)$ such that the squared braiding is the identity on $X\otimes Y$ for all $Y\in \E$).
Then the category $\E'_G$ is naturally equivalent to $\Z(\C)$ as a braided category.
\end{theorem}

\section{The quasi-Hopf algebra $A_q=A_q(\g)$}

In this section we recall the construction of the finite dimensional
basic quasi-Hopf algebras $A_q=A_q(\g)$, given in \cite{eg}, Section
4.

Let $\mathfrak{g}$ be a finite dimensional simple Lie algebra of
rank $r$, and
let $\mathfrak{b}$ be a Borel subalgebra of $\g$.

Let $n\ge 2$ be an odd integer, not divisible by $3$ if $\g=G_2$,
and let $q$ be a primitive root of $1$ of order $n^2$. We will also
assume, throughout the rest of the paper, that $n$ is relatively
prime to the determinant $\det(a_{ij})$ of the Cartan matrix of
$\g$.

Let ${\mathfrak{u}}_q(\mathfrak{b})$ be the Frobenius-Lusztig kernel
associated to ${\mathfrak b}$ (\cite{L1,L2}); it is a finite
dimensional Hopf algebra generated by grouplike elements $g_i$ and
skew-primitive elements $e_i$, $i=1,\dots,r$, such that $$
g_i^{n^2}=1,\;g_ig_j=g_jg_i,\;g_ie_jg_i^{-1}=q^{\delta_{ij}}e_j, $$
$e_i$ satisfy the quantum Serre relations, and $$\Delta(e_i)=e_i\ot
K_i+1\ot e_i,\ K_i:=\prod_j g_j^{a_{ij}}.$$ The algebra
${\mathfrak{u}}_q(\mathfrak{b})$ has a projection onto
$\mathbb{C}[(\mathbb{Z}/n^2\mathbb{Z}) ^r]$, $g_i\mapsto g_i$ and
$e_i\mapsto 0$. Let $B\subset {\mathfrak{u}}_q(\mathfrak{b})$ be the
subalgebra generated by $\{e_i\}$. Then by Radford's theorem
\cite{r}, the multiplication map
$\mathbb{C}[(\mathbb{Z}/n^2\mathbb{Z}) ^r]\ot B\to
{\mathfrak{u}}_q(\mathfrak{b})$ is an isomorphism of vector spaces.
Therefore, $A_q:=\mathbb{C}[(\mathbb{Z}/n\mathbb{Z}) ^r]B\subset
{\mathfrak{u}}_q(\mathfrak{b})$ is a subalgebra. It is generated by
$g_i^n$ and $e_i$, $1\le i\le r$.

Let $\{1_{z}|z=(z_1,\dots,z_r)\in
(\mathbb{Z}/n^2\mathbb{Z}) ^r\}$ be the set of primitive idempotents
of $\Bbb C[(\mathbb{Z}/n^2\mathbb{Z})^r]$ (i.e.,
$1_{z}g_i=q^{z_i}1_{z})$.

Following \cite{g}, for $z,y\in \Bbb Z/n^2\Bbb Z$ let $c(z,y)=q^{-z(y-y')}$, where
$y'$ denotes the remainder of division of $y$ by $n$.

Let
$${\mathbb J}:=\sum_{z,y\in (\mathbb{Z}/n^2\mathbb{Z})^r}\prod_{i,j=1}^{r}c(z_i,y_j)^{a_{ij}}
1_{z}\ot 1_{y}.
$$
It is clear that it is invertible and
$(\varepsilon\ot \id)({\mathbb J})=(\id\ot \varepsilon)({\mathbb
J})=1$. Define a new coproduct
$$
\Delta_{\mathbb J}(z)=\mathbb
J\Delta(z)\mathbb J^{-1}.
$$

\begin{lemma} The elements
$\Delta_{\mathbb J}(e_i)$ belong to $A_q\ot A_q$.
\end{lemma}

\begin{lemma}
The associator $\Phi:=d{\mathbb J}$ obtained by twisting the trivial
associator by ${\mathbb J}$ is given by the formula
$$
\Phi=\sum_{\beta,\gamma,\delta\in (\mathbb{Z}/n\mathbb{Z})^r}\biggl(\prod_{i,j=1}^r
q^{a_{ij}\beta_i((\gamma_j+\delta_j)'-\gamma_j-\delta_j)}\biggr)
\mathbf 1_\beta\otimes \mathbf 1_\gamma\otimes \mathbf 1_\delta,
$$
where
$\mathbf 1_\beta$ are the primitive idempotents
of $\Bbb C[(\Bbb Z/n\Bbb Z)^r]$, $\mathbf 1_\beta g_i^n=q^{n\beta_i}\mathbf 1_\beta$, and
 we regard the components of $\beta,\gamma,\delta$ as
elements of $\mathbb Z$.\footnote{$\mathbf 1_\beta$ should not be confused with $1_z$ that appeared above.}
Thus $\Phi$ belongs to $A_q\otimes
A_q\otimes A_q$.
\end{lemma}

\begin{theorem} The algebra $A_q$ is a quasi-Hopf subalgebra
of ${\mathfrak{u}}_q(\mathfrak{b})^{\mathbb J}$, which has coproduct $\Delta_{\mathbb J}$ and
associator $\Phi$. It is of dimension $n^{{\rm dim}{\mathfrak g}}$.
\end{theorem}

\begin{remark} The quasi-Hopf algebra $A_q$ is not twist equivalent to
a Hopf algebra. Indeed, the associator $\Phi$ is non-trivial since
the $3-$cocycle corresponding to $\Phi$ restricts to a non-trivial
$3-$cocycle on the cyclic group $\mathbb{Z}/n\mathbb{Z}$ consisting
of all tuples whose coordinates equal $0$, except for the $i$th
coordinate. Since $A_q$ projects onto
$(\mathbb{C}[(\mathbb{Z}/n\mathbb{Z})^r],\Phi)$ with non-trivial
$\Phi$, $A_q$ is not twist equivalent to a Hopf algebra.
\end{remark}

\section{Main results}

Let $T:=(\mathbb{Z}/n^2\mathbb{Z})^r$. We have the following well
known result.

\begin{theorem}\label{quandu} The quantum double $D({\mathfrak{u}}_q(\mathfrak{b}))$ of ${\mathfrak{u}}_q(\mathfrak{b})$ is twist
equivalent, as a quasitriangular Hopf algebra,
to ${\mathfrak{u}}_q(\mathfrak{g})\ot \mathbb{C}[T]$.
Therefore,
$$
\mathcal{Z}(\Rep({\mathfrak{u}}_q(\mathfrak{b})))=\Rep({\mathfrak{u}}_q(\mathfrak{g}))\boxtimes \Vect_{T}
$$
as a braided tensor category, where the braiding on $\Rep({\mathfrak{u}}_q(\mathfrak{g}))$
is the standard one, and $\Vect_{T}$ is the category of
$T$-graded vector spaces with the braiding coming from
the quadratic form on $T$ defined by the Cartan matrix of $\g$.
\footnote{Actually, the quadratic form gives the inverse braiding,
but this is not important for our considerations.}
\end{theorem}

\begin{proof}
It is well known (\cite{D}, \cite{CP}) that
$D({\mathfrak{u}}_q(\mathfrak{b}))$ is isomorphic
as a Hopf algebra to $H:={\mathfrak{u}}_q(\mathfrak{g})\ot \mathbb{C}[T]$,
with standard generators $e_i, f_i, K_i\in
{\mathfrak{u}}_q(\mathfrak{g})$ and
$K_i'\in \mathbb C[T]$,
and comultiplication
$$
\Delta_*(e_i)=e_i\otimes K_iK_i'+1\otimes e_i,\
\Delta_*(f_i)=f_i\otimes K_i'^{-1}+K_i^{-1}\otimes f_i
$$
(in fact, this is not hard to check by a direct computation).
Note that the group algebra $\Bbb C[T\times T]$
is contained in $H$ as a Hopf subalgebra
(with the two copies of $T$
generated by $K_i$ and $K_i'$, respectively).
Consider the bicharacter
of $T\times T$ given by the formula
$$
\langle(a,b),(c,d)\rangle=<a,d>,
$$
where $<,>: T\times T\to \Bbb C^*$ is the pairing given by the
Cartan matrix. Consider the twist $J\in \Bbb C[T\times T]^{\otimes
2}$ corresponding to this bicharacter. It is easy to compute
directly that twisting by $J$ transforms the above comultiplication
$\Delta_*$ to the usual ``tensor product'' compultiplication of $H$:
$$
\Delta(e_i)=e_i\otimes K_i+1\otimes e_i,\
\Delta(f_i)=f_i\otimes 1+K_i^{-1}\otimes f_i,
$$
and the same holds for the universal R-matrix
(this computation uses that $K_i'$ are central elements).
This implies the theorem.
\end{proof}

Let $\Gamma\cong (\mathbb{Z}/n\mathbb{Z})^r$
be the $n-$torsion subgroup of $T$.

Our first main result is the following.

\begin{theorem}\label{main1}
The group $\Gamma$ acts on the category $\C=\Rep(A_q)$, and the
equivariantization $\C^\Gamma$ is tensor equivalent to $\Rep({\mathfrak{u}}_q(\mathfrak{b}))$.
\end{theorem}

The proof of Theorem \ref{main1} will be given in the next section.

By Theorem \ref{help2}(i), Theorem \ref{main1}
implies that the category $\Rep(A_q)$ can be conceptually defined
as the de-equivariantization of $\Rep({\mathfrak{u}}_q(\mathfrak{b}))$.

Our second main result is the following.

\begin{theorem}\label{theorem1}
The Drinfeld center $\mathcal{Z}(\Rep(A_q))$ of $\Rep(A_q)$ is braided
equivalent to $\Rep({\mathfrak{u}}_q(\mathfrak{g}))$. Equivalently, the quantum
double $D(A_q)$ of the quasi-Hopf algebra $A_q$ is twist equivalent
(as a quasitriangular quasi-Hopf algebra)
to the small quantum group ${\mathfrak{u}}_q(\mathfrak {g})$.
\end{theorem}

\begin{proof}
Since
$\mathcal{Z}(\Rep({\mathfrak{u}}_q(\mathfrak{b})))=\Rep({\mathfrak{u}}_q(\mathfrak{g}))\boxtimes
\Vect_{T}$ as a braided category, and $\Rep \Gamma\subset \Vect_T$
is a Tannakian subcategory, we have that $\Rep(\Gamma)\subset
\mathcal{Z}(\Rep({\mathfrak{u}}_q(\mathfrak{b})))$ is a Tannakian
subcategory. Moreover, $\Rep \Gamma\subset \Vect_T$ is a Lagrangian
subcategory (i.e, it coincides with its M\"uger centralizer in
$\Vect_T$), so the M\"uger centralizer $\D$ of $\Rep \Gamma$ in
$\mathcal{Z}(\Rep({\mathfrak{u}}_q(\mathfrak{b})))$ is equal to
$\Rep({\mathfrak{u}}_q(\mathfrak{g}))\boxtimes \Rep(\Gamma)$. This
implies that the de-equivariantization $\D_\Gamma$ is
$\Rep({\mathfrak{u}}_q(\mathfrak{g}))$. On the other hand, by
Theorem \ref{main1},
$\Rep({\mathfrak{u}}_q(\mathfrak{b}))=\Rep(A_q)^\Gamma$, so by
Theorem \ref{help2}(ii) we conclude that
$\mathcal{Z}(\Rep(A_q))=\Rep({\mathfrak{u}}_q(\mathfrak{g}))$, as
desired.
\end{proof}

\section{Proof of Theorem \ref{main1}}

Let us first define an action of $\Gamma$ on $\C=\Rep(A_q)$.

For $j=0,\dots,n-1$, $i=1,\dots,r$, let $F_{ij}:\Rep(A_q)\to \Rep(A_q)$ be the functor
defined as follows. For an object $(V,\pi_V)$ in $\Rep(A_q)$,
$F_{ij}(V)=V$ as a vector space, and $\pi_{F_{ij}(V)}(a)=\pi_V(g_i^j a g_i^{-j})$,
$a\in A_q$.

The isomorphism $\gamma_{ij_1,ij_2}:F_{ij_1}(F_{ij_2}(V))\to F_{i,(j_1+j_2)'}(V)$ is given by
the action of
$$
(g_i^n)^{\frac{(j_1+j_2)'-j_1-j_2}{n}}\in A_q,
$$
and
$\gamma_{i_1j_1,i_2j_2}=1$
for $i_1\ne i_2$.

Let us now consider the equivariantization $\C^\Gamma$. By definition, an object
of $\C^\Gamma$ is a representation $V$ of $A_q$ together with a collection of linear
isomorphisms $p_{i,j}:V\to V$, $j=0,\dots,n-1$, $i=1,\dots,r$,
such that
$$
p_{i,j}(av)=g_i^jag_i^{-j}p_{i,j}(v),\; a\in A_q,\; v\in V,
$$
and
$$
p_{i,j_1}p_{i,j_2}=p_{i,(j_1+j_2)'}(g_i^n)^{\frac{-(j_1+j_2)'+j_1+j_2}{n}}.
$$
It is now straightforward to verify that this is the same as a representation
of ${\mathfrak{u}}_q(\mathfrak{b})$, because ${\mathfrak{u}}_q(\mathfrak{b})$ is generated
by $A_q$ and the $p_{i,j}:=g_i^j$ with
exactly the same relations. Moreover, the tensor product of
representations is the same as for ${\mathfrak{u}}_q(\mathfrak{b})^{\Bbb J}$.
Thus $\C^\Gamma$ is naturally equivalent to $\Rep({\mathfrak{u}}_q(\mathfrak{b}))$, as claimed.

This completes the proof of Theorem \ref{main1}.

\end{document}